\documentclass[12pt,a4paper]{article}
\pdfoutput=1
\usepackage{latexsym,amssymb,amsfonts,amsmath,amsthm}
\usepackage[margin=2cm]{geometry}
\usepackage{mdwlist}
\usepackage{amsrefs}
\usepackage[]{algorithm2e}
\usepackage{calrsfs}
\DeclareMathAlphabet{\pazocal}{OMS}{zplm}{m}{n}
\usepackage{graphicx}
\usepackage{caption}
\usepackage{subcaption}

\numberwithin{equation}{section}
\newtheorem{theorem}{Theorem}
\newtheorem{corollary}[theorem]{Corollary}
\newtheorem{lemma}[theorem]{Lemma}

\theoremstyle{definition}

\let\oldproofname=\proofname
\renewcommand{\proofname}{\rm\bf{\oldproofname}}

\makeatletter
\def\imod#1{\allowbreak\mkern10mu({\operator@font mod}\,\,#1)}
\makeatother

\newcommand{\ignore}[1]{}
\newcommand{\es}{\emptyset}
\newcommand{\phii}{\varphi}

\title{Asymptotic size of covering arrays: an application of entropy 
compression}

\author{
Nevena Franceti\'{c} \\School of Mathematical Sciences\\
Monash University \\
Victoria 3800, Australia\\ 
{\it nevena.francetic@monash.edu} 
\and
Brett Stevens \\
School of Mathematics and Statistics\\
Careleton University \\
Ottawa, ON K1S 5B6 Canada \\
{\it brett@math.carleton.ca} \thanks{Research supported by NSERC grant 
249777.}}

\date{January 15, 2015.}

\begin{document}
\maketitle 

\begin{abstract}
 A covering array $CA(N; t,k,v)$ is an $N \times k$ array $A$ whose each cell 
takes a value for a $v$-set $V$ called an alphabet. Moreover, the 
set $V^t$ is contained in the set of rows of every $N \times t$ subarray of 
$A$. The parameter $N$ is called the size of an array and $CAN(t,k,v)$ denotes 
the smallest $N$ for which a $CA(N; t,k,v)$ exists. It is well known that 
$CAN(t,k,v) = {\rm \Theta}(\log_2 k)$~\cite{godbole_bounds_1996}. In this paper 
we derive two upper bounds on $d(t,v)=\limsup_{k \rightarrow \infty} 
\frac{CAN(t,k,v)}{\log_2 k}$ using the algorithmic approach to the Lov\'{a}sz 
local lemma also known as entropy compression.    
\end{abstract}

\section{Introduction}

A covering array $CA(N;t,k,v)$ is an $N \times k$ array $A$ whose cells take 
values from a set $V$ of size $v$ and the set of rows of every $N \times t$ 
subarray of $A$ contains the whole set $V^t$. The parameter $t$ is 
called the strength, the parameter $v$ is the alphabet size and $N$ is called 
the size of the array. A covering array with given parameters $t$, $k$ and $v$ 
always exists. The two central questions regarding covering arrays are: what the 
smallest number of rows is, denoted by $CAN(t,k,v)$, for which a covering 
array with the given set of parameters $(t,k,v)$ exists, and how an array of such size can  
be constructed. In this paper we study the upper bounds on the 
asymptotic size of covering arrays. It is easy to see that if $t=1$ or $v=1$, 
covering arrays are trivial. Hence we assume that $t \geq 2$ and $v\geq 2$. 

Covering arrays are best known for their applications in the software 
testing industry~\cites{kuhn_combinatorial_testing, 
survey_combinatorial_testing} as interaction testing plans. There are numerous
software tools for construction of covering arrays~\cite{www_pairwise}, and 
 there is a vast literature on them as well~\cites{hartman_ca_survey, 
colbourn_ca_survey, binary_cas_survey, kuhn_combinatorial_testing}. However, the 
central question about the optimal size is far from fully answered. The only 
infinite family of covering arrays whose exact size is known is the first 
non-trivial family of arrays of strength $t=2$ and with alphabet size 
$v=2$~\cites{katona, kleitman}. The best known upper bound on the size of a 
covering array for any set of parameters $(t,k,v)$ is obtained by an 
application of the Lov\'{a}sz local lemma~\cite{gargano93}. Together these two 
results give us the asymptotic size of covering  arrays when strength $t$ and 
alphabet size $v$ are fixed and the number of columns $k$ is varied. 

\begin{theorem}~\cites{katona, kleitman, godbole_bounds_1996}\label{thm: log 
growth}
Let $t, v \geq 2$ be integers. Then,
 \[
  CAN(t,k,v) = \rm{\Theta}(\log_2 k).  
 \]
\end{theorem}

Given the previous theorem, there is significant interest in determining the 
following two values (we use the notation given in~\cite{binary_cas_survey}):
\[
\begin{array}{lcr}
 c(t,v) = \liminf_{k \rightarrow \infty} \frac{CAN(t,k,v)}{\log_2 k} & 
\hbox{   and   } &
 d(t,v) = \limsup_{k \rightarrow \infty} \frac{CAN(t,k,v)}{\log_2 k}. \\
 \end{array}
\]

The exact value of $d(t,v)$ is only known when $t=2$. 

\begin{theorem}~\cite{gargano93}\label{thm: gargano d_2_v}
Let $v \geq 2$ be an integer. Then $ d(2,v) = \frac{v}{2}.$ 
\end{theorem}

However, covering arrays which meet this asymptotic size are hard to construct. 
The only family which we currently know how to construct which attains this 
size is the already mentioned family of $CA$s with $t=2$ and 
$v=2$~\cites{katona, kleitman}.

In 1996, Godbole~et.~al.~\cite{godbole_bounds_1996} gave an upper 
bound on $d(t,v)$ for any strength $t \geq 2$. 

\begin{theorem}~\cite{godbole_bounds_1996}\label{thm: old LLL}
Let $t \geq 2$ and $v$ be positive integers. Then,
 \[d(t,v) \leq \frac{(t-1)}{\log_2 \frac{v^t}{v^t-1}}. \]
\end{theorem}

Recently, the method of entropy compression 
was successfully used in the context of 
vertex-colourings of graphs~\cites{vida_nonrepetitive_colouring, 
entropy_compression_2014} to improve on the previous results which used the 
local lemma. In this paper we explore an application of this method in 
the context of covering arrays. We give a new upper bound on $d(t,v)$ for any 
$t \geq 3$ in Theorem~\ref{thm: d_t_v for any t}, which improves 
Theorem~\ref{thm: old LLL}. We also obtain a tighter upper bound on $d(t,v)$ 
given in Lemma~\ref{lemma: multivariable upper bound on d} which depends on 
further computational approximations. Table~\ref{table: constants} displays our 
new upper bounds on $d(t,v)$ for $2 \leq t \leq 6$ and $2 \leq v \leq 10$.  
Finally, we analyze these results and point out possible challenges and further 
avenues for improvement.

\section{Algorithm}\label{sec: algorithm}

We adapt the algorithms given in~\cites{vida_nonrepetitive_colouring, 
entropy_compression_2014} to covering arrays. The algorithm is used as a 
tool for counting. The main idea is to keep a record of  
execution for the algorithm. This allows us to match an input sequence to 
the algorithm injectively with a pair consisting of the 
output array and the record of the execution.  For a given input, we say that 
the execution was 
\emph{unsuccessful} and that it produced a \emph{bad} output  
if, the output array is only partially filled and has 
some empty columns. 
If   the total number of possible input sequences is greater than the 
total number of bad output pairs, then 
there 
must exist an input sequence for which the algorithm successfully 
terminates. Before we give the algorithm, we need to introduce some notation 
which is required for the analysis.  

Given array parameters $N, t, k$ 
and $v$, the algorithm attempts to construct a covering array of size $N 
\times k$ one column at a time. A column of a $CA(N; t,k,v)$ is an element of 
$V^N$, where $V$ is the alphabet set of size $v$. To remind us that these 
ordered $N$-tuples are columns, we denote elements of $V^N$ by $c$. Let 
$\mathcal{I} \subseteq V^N$ denote a set of all  admissible input 
columns for the algorithm. We will define $\mathcal{I}$ in Section~\ref{sec: 
balanced arrays}. Then the algorithm receives as an input value  $I 
\in \mathcal{I}^\ell$, a sequence of 
$\ell$ columns, where $\ell$ is the number of iterations to be performed. 
Let $I(j)$ denote the $j^{\rm{th}}$ coordinate of $I$. 

At an intermediate step in the algorithm, some columns of the array may still 
be empty. Let $\es$ denote an empty column, and let $\mathcal{I}' = \mathcal{I} 
\cup \{ \es \}$. Then, an array can be represented 
as a sequence of $k$ elements of $\mathcal{I}'$, i.e. $A = (c_1, c_2, \dots, 
c_k)$, where $c_i \in \mathcal{I}'$. Let $A(i) = c_i$ be the values in the 
$i^{\rm{th}}$ column of the array. We also require a way to choose which column 
to fill at each step. Define $\phii(A)$ to be the priority function on the 
empty columns of $A$: 
\[
 \phii(A) = \left\{\begin{array}{ll}
		    -1, & \mbox{ if there is no empty column in } A, \\
                    \min\{i \, : \, A(i) = \es \},  &  \rm{otherwise.} \\
                   \end{array}
 \right.
\]
The key property of a covering array is that the subarray on any $t$ columns 
 contains each $t$-tuple in $V^t$ at least once. Let 
$\mathcal{T}$ be the set of all 
$t$-subsets of 
the set $[1,k] = \{1,2,\dots, k\}$. Let $\tau = \{i_1, i_2, \dots, i_t\} \in 
\mathcal{T}$, where 
$i_1 < i_2 < \cdots < i_t$. Then denote by $A|_{\tau}= (c_{i_1}, c_{i_2}, 
\dots c_{i_t})$ the subarray of $A$ on columns indexed by $\tau$. An auxiliary 
function $\mathtt{is\, a \, covering}(A|_{\tau})$ returns true if the set of 
rows of 
$A|_{\tau}$ contains $V^t$ as a subset, and false otherwise.

Now we are ready to describe the algorithm which attempts to construct a 
$CA(N; t,k,v)$ for some positive integers $N$, $t$, $k$ and $v$. It 
starts by initializing all columns of an $N \times k$ array $A$ to be empty, 
and opens a new record file $R$. Then it runs for $\ell$ iterations where 
$\ell$ is the length of the input sequence. The partially constructed array $A$ 
satisfies the covering property at the beginning of each iteration.  
At a step $j$, let $i$ be 
smallest index of an empty column of $A$. The algorithm assigns to the 
$i^{{\rm th}}$ column the $j^{\rm{th}}$ element of the input 
sequence. Now, if $A$ has an $N \times t$ subarray on columns $\tau \subset 
[1,k]$,  which is not a covering, then $i \in \tau$ since $A$ met the covering 
property before algorithm entered the $j^{\rm{th}}$ iteration.
The algorithm 
records $\hat{\tau} = \tau \setminus \{i\}$ and  the content of the 
subarray of $A$ on columns in $\tau$. Note, in order to be able to recover 
input from the output, we need to know the relative position of $i$ with 
respect to other elements in $\hat{\tau}$ since $i$ is not recorded. Hence the 
elements of $\tau$ are first sorted in increasing order. Finally, since 
this subarray does not have the covering property, we assign empty values to 
the 
columns in $\tau$. Otherwise, the addition of a new column to $A$ preserves the 
covering property and the algorithm completes this iteration after recording a 
successful entry to the file.  

Note that the number of lines in the record file $R$ is equal to the number of 
executed iterations. If the algorithm 
completes and the array $A$ has no empty columns, then it is easy to see that 
$A$ 
satisfies the covering property on every set of $t$-columns, i.e. it is a 
covering array. Otherwise, $A$ is 
only partially constructed, it has some empty columns, and we say that the 
execution of the algorithm on the given input was unsuccessful.

\begin{algorithm}[H]
 \KwData{$I \in  \mathcal{I}^{\ell}$ where $\mathcal{I} \subset 
V^N$}
 \KwResult{a (partial) covering array $CA(N;t,k,v)$}
 $A := (\es, \es, \dots, \es)$ \\
 $R:=$ new file()\\
 \For{$j := $ 1 to $\ell$}{
    $i := \phii(A)$ \\
    \If{$i==-1$}{
      break \\
    }\Else{
      $A(i):=I(j)$ \\
      $good :=$ true \\
      \For{all $\tau \in \mathcal{T}$}{
	\If{$i \in \tau$ and $A|_{\tau}$ has no empty columns}{
	  $good := \mathtt{is \, a \, covering}(A|_{\tau})$ \\
	  \If{$good ==$ false}{
	      \begin{tabular}{ll}
	      // omit $i$ from  ${\tau}$ & \\
	      $\mathtt{sort}(\tau)$  & // so that $\tau(r_1)<\tau(r_2)$ 
when $r_1<r_2$ \\
	      $(i_1, i_2, \dots, i_t) := \tau$ & \\
	      $ h := $ index of $i$ in $\tau$ & // i.e. $i_h = i$ \\
	      $\hat{\tau} := (i_1, i_2, \dots, i_{h-1}, i_{h+1}, \dots, i_t)$ & 
// i.e. $\hat{\tau} = \tau \setminus \{i\}$ \\ & \\
	      \end{tabular} \\
	      // record the content of $A|_{\tau}$ and delete these columns 
 \\
	      $(c_1, c_2, \dots, c_t) := (A(i_1), A(i_2), \dots, A(i_t))$  \\
	      $A|_{\tau} := (\es, \es, \dots, \es)$  \\ $\,$ \\
	    
	      $R.\mathtt{write}$(`back-track -- in columns:', $\hat{\tau}$, ` 
deleted content: ', $(c_1, c_2, \dots, c_t)$, `\textbackslash n' )  \\
	      
	    {\bf break} // break the loop over $\tau \in \mathcal{T}$ \\
	}
      }
    }
    \If{$good ==$ true} 
    {
      $R.\mathtt{write}$(`successful entry \textbackslash n') \\
    }
  }  
 }
 \Return{$(A, R)$}

 \caption{Entropy compression algorithm for construction of a 
$CA(N;t,k,v)$} \label{alg:CA2}
\end{algorithm}

\section{Reversibility}\label{sec: bijection}

Next, we establish bijection between the set of all possible inputs 
$\mathcal{I}^\ell$ and the set of all possible outputs $\mathcal{O}_\ell = 
\{(A,R) \, : \, \mbox{obtained by the algorithm on an input } I \in 
\mathcal{I}^\ell\}$. It is easy 
to see that for an input sequence $I \in \mathcal{I}^{\ell}$, we get only one 
output $(A, R)$. We prove the converse in several steps. Let $A_j$ 
denote the state of the array $A$ at the beginning of the 
$j^{\rm{th}}$ iteration of Algorithm~\ref{alg:CA2}. Hence, $A_1$ is an empty 
array, and $A_{\ell+1}=A$, the array returned by the algorithm.

\begin{lemma}\label{lemma: column at step j} Given $(A,R) \in \mathcal{O}_\ell$, 
we can determine the set of indices of all columns 
which are empty in $A_j$ for all $j \in \{1,2, \dots, \ell\}$.
\end{lemma}
\begin{proof}
We use induction on $j$. Denote by $\mathcal{E}_j$ the set of all 
indices of columns which are empty in $A_j$. When $j=1$, $\mathcal{E}_1 = 
[1,k]$, since the algorithm starts with an empty array $A_1$. 

Assume that we know $\mathcal{E}_j$ for some $j<\ell$. Then, $i= \min 
\mathcal{E}_j = \phii(A_j)$ is the index of a column which receives a value 
in the $j^{\rm{th}}$ iteration. If the $j^{\rm{th}}$ line of $R$ starts with 
`successful entry', then $\mathcal{E}_{j+1} = \mathcal{E}_j \setminus \{i\}$. 
Otherwise, the $j^{\rm{th}}$ line of $R$ contains $\hat{\tau} = \tau \setminus 
\{i\}$, where $\tau$ is the set of columns whose content is removed 
at step $j$. Hence, $\mathcal{E}_{j+1} = \mathcal{E}_j \cup \hat{\tau}$. 
\end{proof}

The following is an immediate corollary.
\begin{corollary}\label{cor: phii A_j}
We can determine $\phii(A_j)$ for all $j \in [1,\ell]$ from an output of the 
algorithm $(A,R) \in \mathcal{O}_\ell$ .
\end{corollary}

Next, we determine  $A_j$ at each step of the algorithm from the 
output values.

\begin{lemma}\label{lemma: A_j}
Given $(A,R) \in \mathcal{O}_\ell$, we can deduce $A_j$ for all $j 
\in [1, \ell+1]$.
\end{lemma}
\begin{proof}
The proof is by reverse induction. When $j=\ell+1$, $A_j = A$, the output of 
the 
algorithm. Assume that we know $A_{j+1}$ for some $j < \ell$. By 
Corollary~\ref{cor: phii A_j}, we know $i = \phii(A_j)$. We have two cases 
to consider.  If the $j^{\rm{th}}$ line of $R$ starts with `successful entry', 
then $A_j$ is obtained by deleting the content of column $i$ in $A_{j+1}$. 
Otherwise, the $j^{\rm{th}}$ line of $R$ contains $\hat{\tau}$,  indices of 
all but one of the columns whose content is deleted at step $j$ of the 
algorithm. It also has the content of all $t$ of these columns, 
$(c_{i_1}, c_{i_2}, \dots, c_{i_t}) \in V^t$, where $\hat{\tau} \cup \{i\} = 
\{i_1, i_2, \dots, i_t \}$ such that $i_{r_1} < i_{r_2}$ when $r_1 < r_2$. Then 
$A_j$ 
is obtained from $A_{j+1}$ after the following assignment: $A_{j+1}(i_r) = 
c_{i_r}$ for all $r \in [1,t]$.
\end{proof}

Finally, we are ready to prove the reversibility: given an output, we can 
obtain the unique input sequence for the algorithm.
\begin{lemma}\label{lemma: unique input}
Given $(A,R) \in \mathcal{O}_\ell$, there is a unique input sequence
$I \in \mathcal{I}^\ell$, such that Algorithm~\ref{alg:CA2} produces $(A,R)$ on 
input $I$. 
\end{lemma}
\begin{proof}
The proof is by induction on $\ell$. If $\ell=1$, then $I = A(1)$. Assume that 
the 
statement is true for some $\ell \geq 1$. Let $(A,R) \in \mathcal{O}_{\ell+1}$ 
and denote by $I$ the desired input sequence. Let $R'$ be the record $R$ 
without 
the last line. By Lemma~\ref{lemma: A_j}, we know the value of $A_{\ell+1}$ and 
$(A_{\ell+1}, R') \in \mathcal{O}_\ell$. By our assumption, there is a unique 
input 
sequence $I' \in \mathcal{I}^\ell$ such that the algorithm gives $(A_{\ell+1}, 
R')$ 
on input $I'$. Then $I(j)=I'(j)$ for $j \in [1, \ell]$. It remains 
to determine $I(\ell+1)$.

If the last line of $R$ is `successful entry', then it must be that $I(\ell+1) 
= A(\phii(A_{\ell+1}))$, where $\phii(A_{\ell+1})$ is given by 
Corollary~\ref{cor: phii A_j}.  Otherwise, the last line of $R$ contains 
$\hat{\tau}$ and $(c_1, c_2, \dots, c_t)$. As before, let $\hat{\tau} \cup 
\{\phii(A_{\ell+1})\} = \{i_1, i_2, \dots, i_t \}$, where $i_{r_1} < i_{r_2}$ 
when $r_1 < r_2$. Let $h$ be such that $\phii(A_{\ell+1}) = i_h$. Then we 
have that $I(\ell+1) = c_h$ which is uniquely determined. 
\end{proof}

\section{Algorithm analysis}

In Section~\ref{sec: bijection} we established a bijection between the total 
number of inputs $\mathcal{I}^\ell$ and outputs $\mathcal{O}_\ell$ of 
Algorithm~\ref{alg:CA2}. Next, we want to show that when a given set of 
covering array 
parameters satisfies certain conditions and $\ell$ is big enough, 
the total number of inputs to the algorithm is greater than the set of outputs 
which have exactly $\ell$ lines in the record file (which  correspond to 
unsuccessful executions). Hence, the algorithm will successfully terminate and 
output a covering array with desired parameters for some input sequence.

We start by finding an upper bound on the size of $\mathcal{R}_\ell$, the set 
of all possible record files $R$ with $\ell$ lines which can be output from 
Algorithm~\ref{alg:CA2}. Let $\ell_0$ be the number of `successful entry' 
lines, and $\ell_1$ be the number of `back-track' lines. Then $\ell = 
\ell_0 + \ell_1$ and these lines can be positioned in the record file in ${\ell 
\choose \ell_0} = {\ell \choose \ell_0, \ell_1}$ ways. Denote by $C_1$ the 
number of distinct pairs $(\hat{\tau}, (c_1, c_2, \dots, c_t) \,)$ which can 
appear in a `back-track' line. Then
\[
 |\mathcal{R}_\ell | = \sum_{(\ell_0, \ell_1) \atop \ell_0+\ell_1 = \ell } 
{\ell \choose \ell_0, \ell_1} \; C_1^{\ell_1}.
\]

Now, we can apply the following result from~\cite{entropy_compression_2014}. 

\begin{theorem}\cite{entropy_compression_2014}*{Corollary 19.} \label{thm: B Q}
 Let $\ell$ and $p$ be positive integers. Let $s_i  
\in \mathbb{Z}^+$ and $C_i > 1$, $C_i \in \mathbb{R}$, for $i \in [1,p]$.  
Define $B_\ell$ to be 
\[
 B_\ell(\ell_0, \ell_1, \dots, \ell_p) = {\ell \choose \ell_0, \ell_1, \dots, 
\ell_p} \; \prod_{i=1}^p C_i^{s_i},
\]
where $\ell_i$ is a non-negative integer, $i \in [0,p]$,  $\sum_{i=0}^p 
\ell_i = \ell$ and $\ell \geq \sum_{i=1}^p s_i \ell_i$. Then 
\[
\sum_{(\ell_0, \ell_1, \dots, \ell_p)}  B_\ell(\ell_0, \ell_1, \dots, \ell_p)  
< \ell (\ell + 1)^p \left( \inf_{0 < x \leq 1} Q(x) \right)^\ell,
\]where 
\[
Q(x) = \frac{1}{x} \left( 1 + \sum_{i=1}^p C_i x^{s_i} \right).
\]
\end{theorem}

\begin{corollary}\label{cor: size R_l} 
Let $C_1$ be the  number of distinct pairs $(\hat{\tau}, (c_1, c_2, \dots, c_t) 
\,)$ which can be recorded in a `back-track' line in an execution of 
Algorithm~\ref{alg:CA2}. Then, 
 \[
  |\mathcal{R}_\ell| < \ell (\ell+1) \left( \frac{t}{t-1}(t-1)^{\frac{1}{t}} \;
C_1^{\frac{1}{t}} \right)^\ell. 
 \]
\end{corollary}

\begin{proof}
We apply Theorem~\ref{thm: B Q} with $p=1$ and we only need 
to 
determine the value of $s_1$. Note that the algorithm cannot back-track 
unless there are $t$ non-empty columns in $A$. Since the total number of added 
columns in $A$ is $\ell$, one at each iteration, and the total number of 
deleted columns is $ t \ell_1$, we have that $\ell \geq t \ell_1$. Thus, let 
$s_1 = t$. Now taking the first derivative of $Q(x)$ to get the minimum, the 
result follows.   
\end{proof}

Finally, we give a lemma which is going to be our main tool in further 
analysis.

\begin{lemma}\label{lemma: CA existence}
 Given positive integers $N$, $t$, $k$ and $v$, and a set $\mathcal{I} \subseteq 
V^N$, where $|V|=v$, there exists a $CA(N; t,k,v)$ whose columns are elements 
of the set $\mathcal{I}$ if
\[
 \left(\frac{t}{t-1} \right)^t (t-1) C_1 < |\mathcal{I}|^t,
\]
where  $C_1$ the  number of distinct pairs $(\hat{\tau}, (c_1, c_2, \dots, c_t) 
\,)$ which can be recorded in a `back-track' line in an execution of 
Algorithm~\ref{alg:CA2}.
\end{lemma}

\begin{proof}
Denote by $\overline{\mathcal{O}}_\ell =\{(A,R) \, : \, R \in 
\mathcal{R}_\ell\} \subseteq \mathcal{O}_\ell$, the subset of all possible 
outputs of the algorithm which have exactly $\ell$ lines in the record file 
$R$. Since an output array $A$ has $k$ columns each of which is either 
empty or in $\mathcal{I}$, we have that  
\begin{align*}
|\overline{\mathcal{O}}_\ell |  &\leq (|\mathcal{I}|+1)^k |\mathcal{R}_\ell|  \\
&<  (|\mathcal{I}|+1)^k \ell (\ell+1) \left( \frac{t}{t-1}(t-1)^{\frac{1}{t}} \;
C_1^{\frac{1}{t}} \right)^\ell && \hbox{(by Corollary~\ref{cor: size 
R_l})}\\
&= {\rm O}\left(\left( \frac{t}{t-1}(t-1)^{\frac{1}{t}} \;
C_1^{\frac{1}{t}} \right)^\ell \right) && \hbox{(as a function of $\ell$).}
\end{align*}
Note that $C_1 = C_1(N,t,k,v)$, and hence it is a constant with respect to 
$\ell.$ Therefore, by the assumption of the lemma, for sufficiently large 
$\ell$, we have that $|\overline{\mathcal{O}}_\ell |  < |\mathcal{I}|^\ell$, 
and $|\mathcal{I}|^\ell$ equals the total number of possible inputs of length 
$\ell$ for the algorithm. Since, $|\mathcal{O}_\ell|=  |\mathcal{I}|^\ell$ by 
Lemma~\ref{lemma: unique input}, there exists an input on which the algorithm 
terminates in less than $\ell$ iterations and hence outputs a $CA(N;t,k,v)$.
\end{proof}

In the following section, we apply Lemma~\ref{lemma: CA existence} to 
derive an upper bound on asymptotic size of covering arrays.

\section{Balanced covering arrays of any strength $t$} \label{sec: balanced 
arrays}

To demonstrate how Lemma~\ref{lemma: CA existence} can be applied, we start 
with an easy example for a construction of a covering array of arbitrary 
strength. The main difficulty in the application of 
Lemma~\ref{lemma: CA existence} is to give a good upper bound on the value of 
$C_1$. In Section~\ref{sec: multivar counting}, we will strengthen the general 
result in the cases when $t=2$ and 
$t=3$. 

Recall that $C_1$ equals the number of distinct pairs $(\hat{\tau}, 
(c_1, c_2, 
\dots, c_t))$ which may appear in a `back-track' line in the record file of 
Algorithm~\ref{alg:CA2} for a parameter set $(N; t,k,v)$. The `back-track' line 
is recorded only when the array $(c_1, c_2, 
\dots, c_t)$ is not a proper cover. Hence, 
\[
 C_1 = {k \choose t-1} \cdot |\mathcal{A}_t|,
\]
where $\mathcal{A}_t$ is the set of all $N 
\times t$ arrays on the alphabet set $V$ of size $v$, such that for every array 
in $\mathcal{A}_t$ there is at least one element of $V^t$ which is not 
contained in the set of rows of the array. 

Taking the input set $\mathcal{I} = V^N$ to the equal to the set of all 
possible $N$-tuples on alphabet $V$, we can easily obtain the upper bound on 
the size of a covering array using Lemma~\ref{lemma: CA existence} which is 
almost 
identical to the one derived using 
Lov\'{a}sz local lemma~\cite{godbole_bounds_1996}. This bound is 
improved 
if instead we take $\mathcal{I}$ to be the set of balanced columns: $N$-tuples 
in which every alphabet symbol appears equal number of times. Hence, from now 
on, we will assume that $N = mv$ for some $m$, and $\mathcal{I}$ is the set of 
balanced columns. Therefore, $|\mathcal{I}|= {mv \choose m,m,\dots,m}$. A 
\emph{balanced} covering array, is a covering array whose columns are elements 
of $\mathcal{I}$. 

We also require some approximations of the binomial coefficient which we use in 
the subsequent sections. 

\begin{lemma}\cite{binomial_coeff_bounds}*{Theorems 2.6. and 2.8.}\label{lemma: 
binomial 
bounds}
 Let $m, v \in \mathbb{Z}$, $v \geq 2$ and $m \geq 2$. Then 

\[
l(v) \, m^{-1/2} \, \frac{v^{vm}}{(v-1)^{(v-1)m}}  < {mv \choose m} < 
 u(v) \, m^{-1/2} \, \frac{v^{vm}}{(v-1)^{(v-1)m}},
\] 
where 
\[
 \begin{array}{lcr}
  l(v) =  \frac{e^{15/16}}{\sqrt{2\pi}} \, \left( \frac{v-1}{v} \right)^{(v-1)} 
 & \mbox{   and   } & u(v) = \frac{1}{\sqrt{2\pi}} \, \left( \frac{v}{v-1} 
\right)^{1/2}.\\
 \end{array}
\]
\end{lemma}

\begin{lemma}\cite{korner_information_2011}\label{lemma: binomial entropy}
For any positive integer $n > m$
\[
 {n \choose m} < 2^{n h(\frac{m}{n}),}
\]
where $h(x) = -x \log_2 (x) - (1-x) \log_2(1-x)$ for $0<x<1$.  
\end{lemma}

We will apply Lemma~\ref{lemma: binomial bounds} for parameters $v$ and $m$, 
where $v$ denotes the alphabet size and $m$ denotes the number of occurrences of 
each symbol within a column. Recall that covering arrays are trivial when 
either 
$v=1$ or $t=1$. Also, for any covering array, an obvious lower bound is $N=mv 
\geq v^t$, so $m \geq v^{t-1} \geq v$ for all $t \geq 2$. Hence, the conditions 
of Lemma~\ref{lemma: binomial bounds} always hold for non-trivial parameter 
sets.

Our first application of Lemma~\ref{lemma: CA 
existence} is for the most general case when the strength of a covering 
array is any positive integer $t \geq 2$. 

\begin{theorem}\label{thm: d_t_v for any t}
Let $t$ and $v$ be positive integers, $t,v \geq 2$. Then
\[
d(t,v) \leq \frac{v \, (t-1)}{\log_2 \left(\frac{v^{t-1}}{v^{t-1}-1} \right)}.
\]
\end{theorem}
\begin{proof}
Let $V$ be the alphabet set. Let $k \geq t$ and $m$ be positive 
integers and $\mathcal{I} 
\subset V^{mv}$ be the set of balanced columns.  Since 
$C_1 = {k \choose t-1} 
|\mathcal{A}_t|$, if $m$ is such that 
\begin{align}\label{eq: inequality for lemma 8}
 \frac{\left(\frac{t}{t-1} \right)^t (t-1) \cdot k^{t-1} \cdot 
|\mathcal{A}_t|}{ |\mathcal{I}|^t} &< 1,
\end{align}
then by Lemma~\ref{lemma: CA existence}, there exists a balanced $CA(mv; 
t,k,v)$. 

Now, 
\[
 |\mathcal{A}_t| < v^t \cdot |\mathcal{I}| \cdot (v^{t-1} - 1)^m  \cdot
(v^{t-1})^{m(v-1)}.
\]
Indeed, if $A \in \mathcal{A}_t$, then the following properties hold. 
\begin{itemize*}
 \item There are $v^t$ choices for an element $(a_1,a_2, \dots, 
a_{t}) \in V^t$ which is not covered by 
rows of $A$.
 \item The first column of $A$ can be any element of the input set 
$\mathcal{I}$. 
 \item The $m$ rows of $A$ having $a_1$ in the first column cannot contain the 
ordered $(t-1)$-tuple $(a_2, a_3, \dots, a_{t})$ in the remaining cells. 
 \item All other rows of the array obtained from $A$ by removing the first 
column can contain any element of $V^{t-1}$. 
\end{itemize*}
By Lemma~\ref{lemma: binomial bounds}, 
\begin{align*}
 |\mathcal{I}| = {mv \choose m,m, \dots, m} = {mv \choose m}{m(v-1) \choose m} 
\cdots {2m \choose m} > \left( \prod_{i=2}^{v} l(i) \right) \, 
m^{-\frac{v-1}{2}} \, v^{vm},
\end{align*}
hence
\begin{align}\label{eq 2}
  \frac{\left(\frac{t}{t-1} \right)^t (t-1) \cdot k^{t-1} \cdot 
|\mathcal{A}_t|}{ |\mathcal{I}|^t} &< 
M(v,t) \, k^{t-1} \, m^{\frac{(v-1)(t-1)}{2}} \left(\frac{v^{t-1}-1}{v^{t-1}}  
\right)^m,  
\end{align}
where $M(v,t) = \left(\frac{t}{t-1} \right)^t (t-1) v^t \left(\prod_{i=2}^{v} 
l(i) \right)^{1-t}.$

For fixed covering array parameters $(t,k,v)$, the right hand size of 
inequality~(\ref{eq 2}) is a function of $m$ and its dominant term is 
exponential with base smaller than $1$. Let $m$ be the smallest positive integer 
for which the right hand side of 
inequality~(\ref{eq 2}) is smaller than 1. Then inequality~(\ref{eq: inequality 
for lemma 8}) is satisfied, and so there exists a balanced $CA(mv; t,k,v)$. 
Since $m$ is the smallest such integer, it follows that inequality~(\ref{eq: 
inequality 
for lemma 8}) does not hold for $m-1$, that is
\begin{align*}
 M(v,t) \, k^{t-1} \, (m-1)^{\frac{(v-1)(t-1)}{2}} 
\left(\frac{v^{t-1}-1}{v^{t-1}}  
\right)^{(m-1)} &\geq 1.
\end{align*}
Taking the logarithm of both sides, we get 
\begin{align*}
 \limsup_{k \rightarrow \infty} \frac{m}{\log_2 k} \leq \frac{t-1}{\log_2 
\left( \frac{v^{t-1}}{v^{t-1}-1} \right)}.
\end{align*}
Note that  $\lim_{k \rightarrow \infty} \frac{\log_2 m}{\log_2 k} = 0$ by 
Theorem~\ref{thm: log growth}. Finally, since $CAN(t,k,v)$ is at most the 
size of a balanced $CA(t,k,v)$, we get an upper bound on $d(t,v)$.
\end{proof}

\section{Tighter bound on $d(t,v)$}\label{sec: multivar counting}

The main difficulty in computing the value of $C_1$ is counting the $N \times 
t$ arrays over an alphabet set $V$ which are not covering arrays. We can obtain 
a 
multivariable function in $t-2$ variables to approximate $C_1$ from above. When 
$t=2$ and $3$, we get exact bounds, and for higher values of $t$ we obtain 
these bounds using mathematical software for non-linear 
optimization. 

For the purposes of the following lemma, let $f_{t,v}$ be the following 
function on domain $(0,1)^{t-1}$:
\begin{align*}
f_{t,v} (x_1,x_2, \dots, x_{t-1}) &= \log_2 
\frac{(v-x_{t-1})^{(v-x_{t-1})}}{(v-1-x_{t-1})^{(v-1-x_{t-1})} 
x_{t-1}^{x_{t-1}}} + \\
 &+ \sum_{i=1}^{t-2} 
\log_2 \left(\frac{(v-x_i)^{(v-x_i)}}{(v-1-x_i+x_{i+1})^{(v-1-x_i+x_{i+1})} 
(x_i - x_{i+1})^{(x_i - x_{i+1})} (1-x_{i+1})^{(1-x_{i+1})}}\right).
\end{align*}

\begin{lemma}\label{lemma: multivariable upper bound on d}
 Let $t \geq 2$ and $v$ be positive integers and \[f_0(t,v) = \max_{1=x_1 \geq 
x_2 \geq \cdots \geq x_{t-1} \geq 0} f_{t,v} (x_1, x_2, \dots, x_{t-1}).\] Then 
$d(2,2) = 1$ and when $tv > 4$,
\[
  d(t,v) \leq \frac{(t-1)v}{(t-1) \left( \log_2 \frac{v^v}{(v-1)^{v-1}} 
\right) - f_0(t,v)}.\]
\end{lemma}

\begin{proof}
 As before, let $V$ be the alphabet set of size $v$. Let $k$ be an 
integer, $k \geq t$. We need to bound the size of $\mathcal{A}_t$. A set of 
rows of $A  \in \mathcal{A}_t$ does not contain a $t$-tuple in $V^t$, which we 
denote by $(a_1, a_2, \dots, a_t) \in V^t$. Next we count the number of 
occurrences of the 1-tuple $(a_1)$ in the first column of $A$, the number of 
occurrences of the $2$-tuple $(a_1, a_2)$ 
in the 
first two columns of $A$, and so on. Let $0 \leq x_{i} \leq 1$ be such that the 
subarray of $A$ restricted to columns $1$ through $i$ contains exactly $m
x_{i}$ rows $(a_1, a_2, \dots, a_i)$, where $i \in [1,t]$.  We know that 
$x_1=1$ 
and $x_t=0$ since the columns of $A$ are balanced and does not cover $(a_1, a_2, 
\dots, a_i)$.  Also, note that $x_i 
\geq x_{i+1}$ for all $i$. 

The first column of $A$ can be 
chosen arbitrarily. Any other column $i \geq 2$, contains $mx_i$ cells with 
value $a_i$ within $mx_{i-1}$ rows which contain $(a_1, a_2, \dots, a_{i-1})$ in 
the previously chosen columns. Hence, the $i^{\rm{th}}$ column of $A$ can be 
completed in at most ${mx_{i-1} \choose mx_{i}} {m(v-x_{i-1}) \choose 
m(1-x_i)}{m(v-1) \choose m,m,\dots, m}$ ways.

If $(t,v) \neq (2,2)$, using Lemmas~\ref{lemma: binomial 
bounds}~and~\ref{lemma: binomial 
entropy} , we get
\begin{align} \nonumber
 \frac{\left(\frac{t}{t-1} \right)^t (t-1) \cdot k^{t-1} \cdot 
|\mathcal{A}_t|}{ |\mathcal{I}|^t} &<  \left(\frac{t}{t-1} \right)^t (t-1) 
\cdot k^{t-1} \frac{ \left( \prod_{i=2}^{t-1} {mx_{i-1} \choose 
mx_i}{m(v-x_{i-1} \choose m(1-x_i)}\right) {m(v-x_{t-1}) \choose m}}{{mv 
\choose m}^{t-1}} \\ \nonumber
 & < \left(\frac{t}{t-1} \right)^t (t-1) 
\cdot k^{t-1} \frac{2^{m f_{t,v}(x_1,x_2, \dots, x_{t-1})}}{l(v)^{t-1} 
m^{-(t-1)/2} \left( \frac{v^v}{(v-1)^{v-1}} \right)^{m(t-1)}} \\
 &< \frac{\left(\frac{t}{t-1} \right)^t (t-1) m^{(t-1)/2}}{l(v)^{t-1} } 
\cdot k^{t-1} \left( \frac{2^{ f_0(t,v)}}{
 \left( \frac{v^v}{(v-1)^{v-1}} \right)^{(t-1)}} \right)^m,  \label{eq: 
multivar ineq bound}
\end{align}
where 
\[
f_{t,v}(x_1, \dots, x_{t-1}) = \left( \sum_{i=2}^{t-1} x_{i-1} 
h\left(\frac{x_i}{x_{i-1}}\right) 
+(v-x_{i-1})h\left(\frac{1-x_i}{v-x_{i-1}}\right)\right) + (v-x_{t-1}) 
h\left(\frac{1}{v-x_{t-1}}\right).
\]
In the last inequality, the dominant term is an exponential 
function of $m$. Following the same reasoning as in the proof 
of 
Theorem~\ref{thm: d_t_v for any t}, we get an upper bound on $d(t,v)$.

Using the definition of the entropy function $h$, one can write $f_{t,v}$ in 
the form given above. Also note that $x_1=1$, so it is a dummy variable for 
$f_{t,v}$. 

If $(t,v)=(2,2)$, since $x_1=1$ and $x_2=0$, there is  ${m \choose 0} 
{m \choose 0}{m \choose m} = 1$ choice for the second column and hence 
$|\mathcal{A}_t|=|\mathcal{I}|$. Note that this is the only case for which we 
get the exact count of the number of $N \times t$ arrays which are not  
coverings of strength $t$. Using Lemma~\ref{lemma: binomial bounds}, 
\begin{align*}
 \frac{\left(\frac{t}{t-1} \right)^t (t-1) \cdot k^{t-1} \cdot 
|\mathcal{A}_t|}{ |\mathcal{I}|^t} &<  \frac{4k}{{2m \choose m}} < 
\frac{4m^{1/2} \cdot k}{l(2) \cdot 2^{2m}}.
\end{align*}
As before, taking the smallest $m$ for which the right hand-side of the last 
inequality is smaller than $m$,it follows that $d(2,2) \leq 1$, which is the 
exact 
value of $d(2,2)$~\cites{kleitman, katona}.
\end{proof}

Observe that $f_{2,v}$ is a constant function since $x_1=1$, and $f_{3,v}$ is a 
single 
variable function so we can easily obtain its maximum taking the first 
derivative of $f_{3,v}$. The same result can be obtained using Lov\'{a}sz local 
lemma directly~\cite{ca_asymptotics_godbole}.   

\begin{corollary}\label{cor: d_2_v}
 Let $v$ be a positive integer, $v \geq 2$. Then $d(2,2)=1$ and 
 \[
  d(2,v) \leq \frac{v}{\log_2 \left( \frac{v^v 
(v-2)^{v-2}}{(v-1)^{2(v-1)}} \right)}, \mbox{ when } v \geq 3.
 \]
\end{corollary}

\begin{corollary}\label{cor: d_3_v}
 Let  $v \geq 2$ be an integer. Then 
\[
d(3,v) \leq \frac{2v}{\log_2 \left(\frac{v^{2v} 
(v-1-\xi)^{(v-1-\xi)} \xi^\xi (v-2-\xi)^{(v-2-\xi)} (1-\xi)^{2(1-\xi)} 
}{(v-1)^{3(v-1)} 
(v-\xi)^{(v-\xi)}} \right)},
\] 
where $\xi = \frac{1}{2}(1+v - \sqrt{v^2+2v-3}).$
\end{corollary}
\begin{proof}
The function 
\[f_{3,v}(1,x_2) =  \log_2 \frac{(v-x_2)^{(v-x_2)}}{(v-1-x_2)^{(v-1-x_2)} 
x_2^{x^2}} + \log_2 \frac{(v-1)^{(v-1)} }{(v-2+x_2)^{(v-2+x_2)} 
(1-x_2)^{2(1-x_2)}} \]
is maximum at $\xi = \frac{1}{2}(1+v - \sqrt{v^2+2v-3}) < 1$. It is 
straightforward to apply Lemma~\ref{lemma: multivariable upper bound on d}.
\end{proof}

For $t \geq 4$, $f_{t,v}$ is a multivariable function. We used a successive 
quadratic programming solver in Octave to compute $f_0(t,v)$. 
Table~\ref{table: constants} gives values of $d(t,v)$ obtained in 
Corollaries~\ref{cor: d_2_v}~and~\ref{cor: d_3_v} and by 
computational optimization for $4 
\leq t \leq 6$.

\begin{table}[!h]
\centering
\begin{tabular}{r|rrrrr}\hline
$v$ \textbackslash $t$ &  \multicolumn{1}{c}{2} &  \multicolumn{1}{c}{3} &  
\multicolumn{1}{c}{4} &  \multicolumn{1}{c}{5} &  \multicolumn{1}{c}{6}\\ 
\hline\hline
 \multicolumn{1}{c|}{2} & 1 & 7.56 & 27.32 & 79.74 & 209.13 \\     
 \multicolumn{1}{c|}{3} & 3.97 & 32.03 & 158.65 & 658.21 & 2503.83\\ 
 \multicolumn{1}{c|}{4} & 8.16 & 81.35 & 518.55 & 2816.81 & 14162.67\\ 
 \multicolumn{1}{c|}{5} & 13.72 & 163.91 & 1281.78 & 8635.15 & 54108.77\\ 
 \multicolumn{1}{c|}{6} & 20.65 & 288.03 & 2672.98 & 21523.56 & 161643.64\\ 
 \multicolumn{1}{c|}{7} & 28.98 & 462.05 & 4966.64 & 46555.89 & 407676.24\\ 
 \multicolumn{1}{c|}{8} & 38.68 & 694.28 & 8487.15 & 90802.26 & 908447.35\\ 
 \multicolumn{1}{c|}{9} & 49.78 & 993.05 & 13608.84 & 163661.74 & 1841749.21\\ 
 \multicolumn{1}{c|}{10} & 62.25 & 1366.68 & 20755.89 & 277195.09 & 
3465640.41\\ 
\hline
\end{tabular}
\caption{Upper bounds on $d(t,v)$. }\label{table: 
constants}
\end{table}

\section{Analysis of results}

Theorem~\ref{thm: d_t_v for any t} provides a new upper bound on $d(t,v)$ for 
any $t$. This bound is an improvement on the current best general upper bound 
on $d(t,v)$ derived in~\cite{godbole_bounds_1996}. To see this, recall that 
$\ln \left( 1+\frac{1}{x} \right) = \frac{1}{x} - \frac{1}{2x^2} + {\rm o} 
(\frac{1}{x^2}) = \frac{1}{x+\frac{x}{2x-1}}  + {\rm o} 
(\frac{1}{x^2}) \approx \frac{1}{x+ \frac{1}{2}}$ for $|x| \gg 1$. Hence, for 
a fixed 
value of $t$, 
\begin{align*}
\frac{(t-1)}{\log_2 \frac{v^t}{v^t-1}} &\approx (t-1)(v^t-\frac{1}{2}) \ln(2) 
&& \mbox{ and } 
&&
\frac{(t-1)v}{\log_2 \left( \frac{v^{t-1}}{v^{t-1}-1} \right)} \approx (t-1) 
\left(v^t - \frac{v}{2} \right) \ln(2).\\ 
\end{align*}

The better upper bound on $|\mathcal{A}_t|$ obtained in Section~\ref{sec: 
multivar 
counting} yields the most improvement when $t=2$ since over-counting is the 
least in this case. As above, we can easily approximate the bound obtained in 
Corollary~\ref{cor: d_2_v} to get  
\begin{align*}
\frac{v}{\log_2 \left( \frac{v^v(v-2)^{v-2}}{(v-1)^{2(v-1)}} \right)} =
\frac{v \ln(2)}{v \ln\left( \frac{v}{v-1} \right) - (v-2)\ln\left( 
\frac{v-1}{v-2} \right) } \approx \frac{v(v-\frac{1}{2})(v-\frac{3}{2})}{(v-1)} 
\ln(2) < v(v-1)\ln(2). 
\end{align*}

Hence, we get a tighter bound on $d(2,v)$. However, note that the upper bound 
on $d(2,v)$ for $v \geq 3$ 
given in Corollary~\ref{cor: d_2_v} still quadratic in $v$, which is the same 
as 
the bound given in Theorem~\ref{thm: d_t_v for any t}.  Recall, $d(2,v) = 
\frac{v}{2}$~\cite{gargano93}. Hence, even for the strength is $t=2$, 
the 
obtained upper bound on $d(2,v)$ is far from optimal. However, 
Algorithm~\ref{alg:CA2} provides one 
major improvement to previous asymptotic constructions: when $t=2$ and $v=2$ we 
are able to compute the exact size of $\mathcal{A}_t$, which gives us that 
$d(2,2)=1$ in Corollary~\ref{cor: d_2_v}. This indicates that 
Algorithm~\ref{alg:CA2} might potentially yield asymptotically optimal covering 
arrays. But the current approximation the size of $\mathcal{A}_t$, the set of 
$N \times t$ arrays with balanced columns which are not coverings, introduces 
substantial overcounting even in the easiest case when strength $t=2$. To 
see this in a different way, consider the examples of upper bounds on $d(2,v)$ 
given in Table~\ref{table: t=2 comparison}. We can see the improvements on the 
upper bounds on $d(2,v)$ obtained in Theorem~\ref{thm: d_t_v for any t} and 
Corollary~\ref{cor: d_2_v} compared to Theorem~\ref{thm: old LLL}. The fourth 
row of Table~\ref{table: t=2 comparison} corresponds to a bound obtained by the 
following simple construction. Let $\mathcal{V}$ be a collection of all 
$2$-subsets of an alphabet set $V$ of size $v$. Then a $CA(2,k,v)$ on alphabet 
set $V$ can be constructed by juxtaposing ${v \choose 2}$ isomorphic copies of 
a $CA(2,k,2)$ on alphabet set $V'$ for every $V'\in \mathcal{V}$. Since 
$d(2,2)=1$, we get  $d(2,v) \leq \frac{v(v-1)}{2} < v(v-1)\ln(2)$, giving 
improvement 
to the general bound on $d(2,v)$ obtained by Corollary~\ref{cor: d_2_v}. More 
advanced direct constructions of covering arrays of strength $t=2$, especially 
when $v$ is a prime power, provide covering arrays which yield even smaller  
bounds on $d(2,v)$ which are still quadratic in $v$ (for example, 
see~\cite{mixedColbourn}). The fifth row of Table~\ref{table: t=2 comparison} 
gives the slope of least square regression line for the set of pairs $(\log_2 
k, N)$ such that $N$ is the smallest size for which a $CA(N;2,k,v)$ is 
currently known  (as given in tables in~\cite{ca_tables_colbourn}). We 
can see that these 
values are still far 
away from the optimal asymptotic size given in the last row of 
Table~\ref{table: 
t=2 comparison}, with the exception of $v=2$.

\begin{table}[!h]
\centering
\begin{tabular}{c | rrrrrrrrr}\hline
$d(2,v)$ \textbackslash $v$ & 2 & 3 & 4 & 5 & 6 & 7 & 8& 9 & 10 \\ \hline \hline
Theorem~\ref{thm: old LLL} & 2.41 & 5.89 & 10.74 & 16.98 & 24.61 & 33.62 & 
44.01 & 55.80 & 68.97 \\
Theorem~\ref{thm: d_t_v for any t} & 2.0 & 5.13& 9.64 
& 15.53 & 22.81 & 31.48 & 41.53 & 52.96 & 65.79 \\
Corollary~\ref{cor: d_2_v} & 1  & 3.97  & 8.16 & 13.72 & 20.65  & 28.98 
& 38.68  &  49.78 & 62.25 \\ \vspace{6pt}
${v \choose 2}$ & 1  & 3  & 6  & 
10  & 15  & 21  & 28  & 
36 & 45 \\ \hline
slope of regression & 1.02 & 2.84 & 5.15 & 7.935 & 11.83 & 15.49 & 
19.55  & 21.99 & 25.83 \\
Theorem~\ref{thm: gargano d_2_v} & 1 & 1.5 & 
2  & 2.5  & 3  & 3.5  & 
4  & 4.5 & 5 \\\hline
\end{tabular}
\caption{Comparison of upper bounds on $d(2,v)$. }
\label{table: t=2 comparison}
\end{table}

We have seen  that Theorem~\ref{thm: d_t_v for any t} provides an 
improvement on the upper bound for $d(t,v)$ compared to the 
current best known result stated in Theorem~\ref{thm: old LLL} for any value 
of $t$. However, the improvement 
obtained is comparatively small as $t$ increases (for example, see 
Table~\ref{table: t=6 comparison}). On the other hand, the 
 upper bounds obtained here predict the existence 
of covering arrays with smaller size that what is currently known. Indeed, 
Algorithm~\ref{alg:CA2} terminates and outputs a proper covering array 
when~(\ref{eq: inequality for lemma 8}) is satisfied. That means that for a 
given $t$, $k$ and $v$, if $m$ is such that the value in~(\ref{eq: multivar 
ineq bound}) is smaller than 1, a $CA(vm; t,k,v)$ exits.  Figure~\ref{fig: t=6 
comparison} plots the current best known sizes of covering arrays with $t=6$, 
and $v=2$ or $v=7$ given in~\cite{ca_tables_colbourn} against the sizes of 
covering arrays for which~(\ref{eq: multivar ineq bound}) is smaller than 1. We 
can see that for small values of $k$, the current, predominately computational 
results, are producing covering arrays of smaller size. However, for large 
values of $k$ we are predicting the existence of covering arrays with much 
smaller number of rows.   

\begin{table}[!h]
\centering
\begin{tabular}{c | rrrrrrrr}\hline
$d(6,v)$ \textbackslash v & 2 & 3 & 4 & 5 & 6 & 7  \\ \hline \hline
Theorem~\ref{thm: old LLL} & 220.07 & 2524.79 & 14193.92 &
54150.39 & 161695.64 & 407738.63  \\
Theorem~\ref{thm: d_t_v for any t} & 218.32 & 2521.32 & 
14188.72 &
54143.46 & 161686.98 & 407728.23  \\
Table~\ref{table: constants} & 209.13 
 & 2503.83 & 14162.67 & 54108.77 & 161643.64 & 407676.24  \\ \hline
\end{tabular}
\caption{Comparison of upper bounds on $d(6,v)$. }
\label{table: t=6 comparison}
\end{table}

\begin{figure}[!h] 
\centering
\begin{subfigure}{.5\textwidth}
  \centering
  
\includegraphics[width=\linewidth]{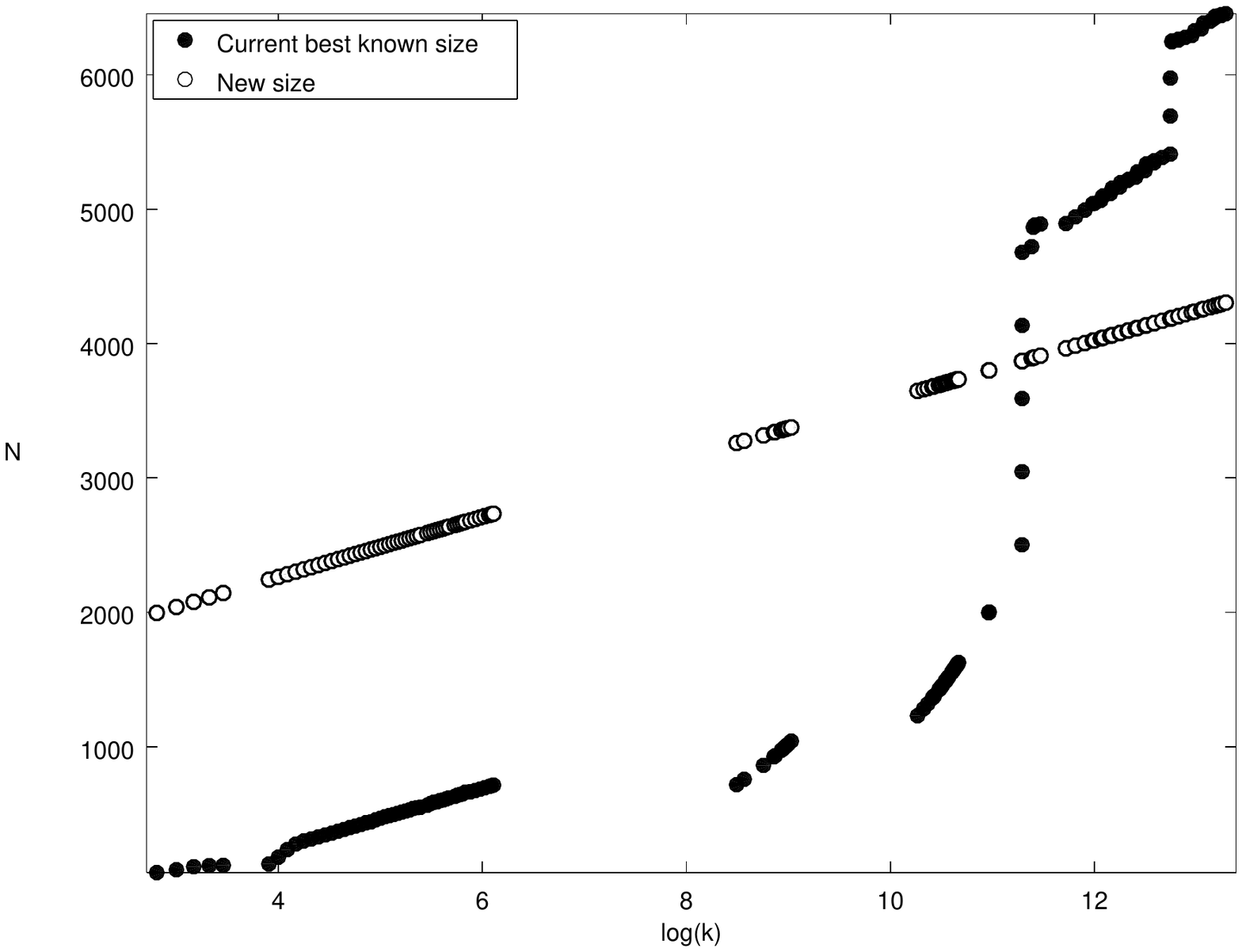}
  \caption{$CA$s with $t=6$ and $v=2$.}
  \label{fig:sub1}
\end{subfigure}%
\begin{subfigure}{.5\textwidth}
  \centering
  
\includegraphics[width=\linewidth]{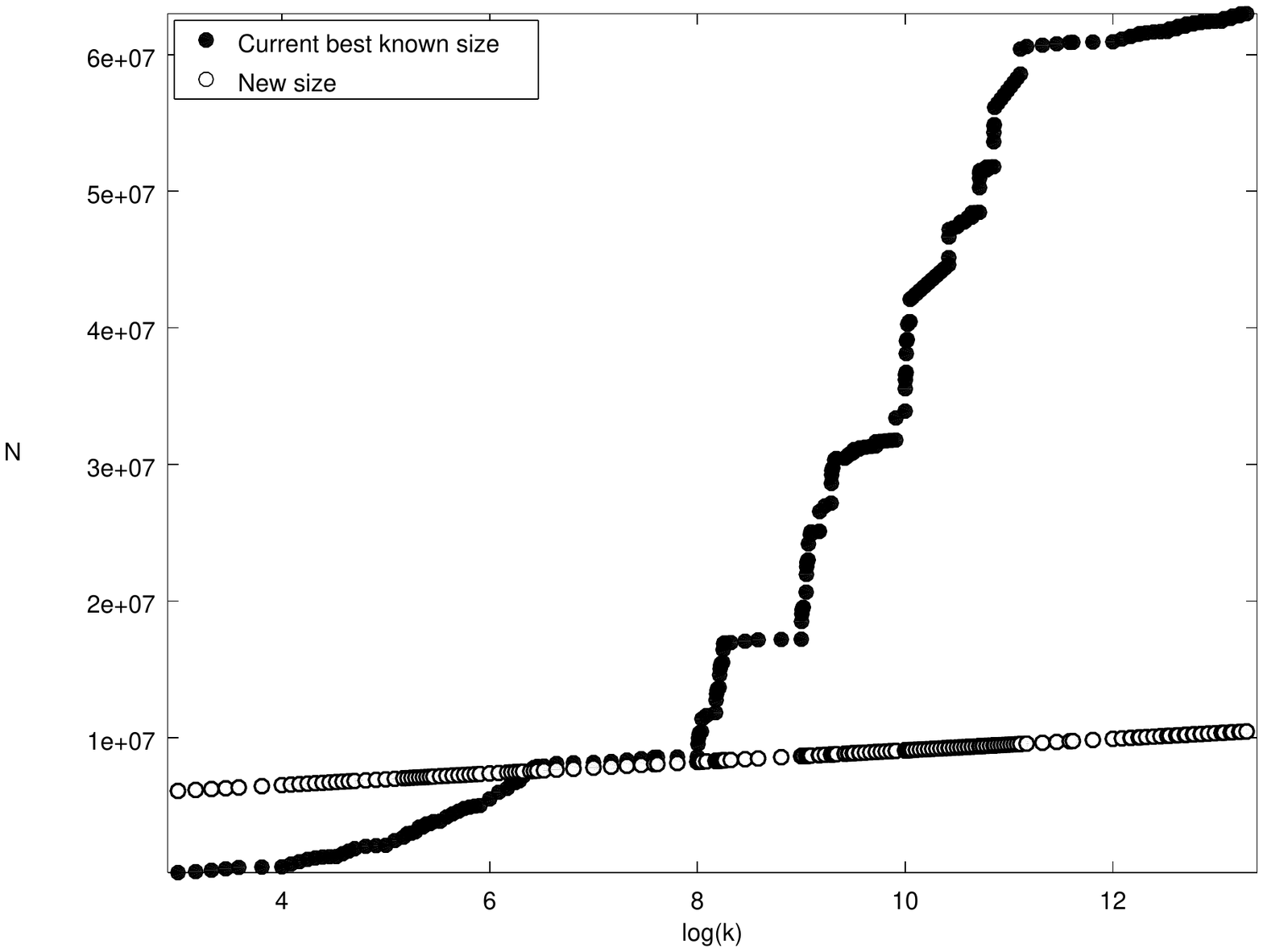}
  \caption{$CA$s with $t=6$ and $v=7$.}
  \label{fig:sub2}
\end{subfigure}
\caption{Comparison of sizes of covering arrays which can be constructed by 
Algorithm~\ref{alg:CA2} with currently best known sizes.}
\label{fig: t=6 comparison}
\end{figure}

\begin{section}{Conclusion}
 
 Determining the optimal size of a covering array for a given triple $(t,k,v)$ 
and constructing optimal covering arrays have been two central questions in 
this area of research. The interest in these two questions stems from the fact 
that covering arrays are natural models for interaction test suites and hence 
they are extensively used in the blooming software testing industry. However, 
these two questions have proven to be a great challenge for both combinatorial 
and 
computer science research communities. 

In this paper we tackled the problem of determining the upper bounds on the 
asymptotic size of covering arrays using an algorithmic version of the local 
lemma. 
We determined a new general bound on $d(t,v)$ (see Theorem~\ref{thm: d_t_v for 
any t}) and we gave a tighter bound in Lemma~\ref{lemma: multivariable upper 
bound on d} which depends on further numerical computation. 

However, though we are improving the existing upper 
bounds on the asymptotic size of covering arrays for strength $t \geq 3$, in 
the simplest case when $t=2$ (and the over-counting is the least), the 
bounds we are obtaining are far from the optimal predicted by Theorem~\ref{thm: 
gargano d_2_v}. The main challenge in improving 
these bounds is finding a better way to count the number of balanced arrays on 
$t$ columns which are not $t$-coverings. A new view to this problem may lead to 
better encrypting of information in the `back-track' lines in the 
algorithm. Indeed, in the 
case when $t=2$ and $v=2$, we are able to count these arrays exactly and as a 
result this general algorithm produces covering arrays whose size is 
asymptotically optimal. 

\end{section}

\begin{section}{Acknowledgments}
 We sincerely thank Marni Mishna and Steve Melczer for many insightful 
discussions on enumeration techniques.   
\end{section}



\begin{bibdiv}
\begin{biblist}
\bib{ca_tables_colbourn}{misc}{
      author={Colbourn, Charles},
       title={Covering array tables for $t=2,3,4,5,6$},
        date={2014},
         url={http://www.public.asu.edu/~ccolbou/src/tabby/catable.html},
        note={Accessed on Aug 16, 2014.},
}

\bib{colbourn_ca_survey}{article}{
      author={Colbourn, Charles~J},
       title={Combinatorial aspects of covering arrays},
        date={2004},
     journal={Le Matematiche},
      volume={59},
      number={1-2},
       pages={125\ndash 172},
}

\bib{mixedColbourn}{article}{
      author={Colbourn, Charles~J},
      author={Martirosyan, Sosina~S},
      author={Mullen, Gary~L},
      author={Shasha, Dennis},
      author={Sherwood, George~B},
      author={Yucas, Joseph~L},
       title={Products of mixed covering arrays of strength two},
        date={2006},
     journal={Journal of Combinatorial Designs},
      volume={14},
      number={2},
       pages={124\ndash 138},
}

\bib{korner_information_2011}{book}{
      author={Csisz\'{a}r, Imre},
      author={K\"{o}rner, J\'{a}nos},
       title={Information theory: Coding theorems for discrete memoryless
  systems},
     edition={2},
   publisher={Cambridge University Press},
        date={2011},
}

\bib{www_pairwise}{misc}{
      author={Czerwonka, Jacek},
       title={Pairwise testing: Combinatorial test case generation},
        date={2014},
         url={www.pairwise.org},
        note={Accessed on Oct 23, 2014.},
}

\bib{vida_nonrepetitive_colouring}{article}{
      author={Dujmovi\'{c}, Vida},
      author={Joret, Gwena\"{e}l},
      author={Kozik, Jakub},
      author={Wood, David~R.},
       title={Nonrepetitive colouring via entropy compression},
        date={accepted in 2013},
     journal={Combinatorica},
}

\bib{gargano93}{article}{
      author={Gargano, L.},
      author={K\"{o}rner, J.},
      author={Vaccaro, U.},
       title={Sperner capacities},
        date={1993},
     journal={Graphs and Combinatorics},
      volume={9},
      number={1},
       pages={31\ndash 46},
}

\bib{godbole_bounds_1996}{article}{
      author={Godbole, Anant~P},
      author={Skipper, Daphne~E},
      author={Sunley, Rachel~A},
       title={$t$-covering arrays: upper bounds and poisson approximations},
        date={1996},
     journal={Combinatorics, Probability and Computing},
      volume={5},
      number={2},
       pages={105{\textendash}117},
}

\bib{entropy_compression_2014}{article}{
      author={Gon\c{c}alves, Daniel},
      author={Montassier, Micka\"{e}l},
      author={Pinlou, Alexandre},
       title={Entropy compression method applied to graph colorings},
        date={2014-06-17},
      eprint={arXiv:1406.4380},
         url={http://arxiv.org/abs/1406.4380},
        note={Accessed on July 13, 2014.},
}

\bib{hartman_ca_survey}{article}{
      author={Hartman, Alan},
      author={Raskin, Leonid},
       title={Problems and algorithms for covering arrays},
        date={2004},
     journal={Discrete Mathematics},
      volume={284},
      number={1-3},
       pages={149\ndash 156},
}

\bib{katona}{article}{
      author={Katona, G. O.~H.},
       title={Two applications (for search theory and truth functions) of
  sperner type theorems},
        date={1973},
     journal={Periodica Mathematica Hungarica. Journal of the J\'{a}nos Bolyai
  Mathematical Society},
      volume={3},
       pages={19\ndash 26},
}

\bib{kleitman}{article}{
      author={Kleitman, Daniel~J.},
      author={Spencer, Joel},
       title={Families of $k$-independent sets},
        date={1973},
     journal={Discrete Mathematics},
      volume={6},
       pages={255\ndash 262},
}

\bib{kuhn_combinatorial_testing}{book}{
      author={Kuhn, D.~Richard},
      author={Kacker, Raghu~N.},
      author={Lei, Yu},
       title={Introduction to combinatorial testing},
     edition={1},
   publisher={Chapman \& Hall/{CRC}},
        date={2013},
}

\bib{binary_cas_survey}{article}{
      author={Lawrence, Jim},
      author={Kacker, Raghu~N.},
      author={Lei, Yu},
      author={Kuhn, D.~Richard},
      author={Forbes, Michael},
       title={A survey of binary covering arrays},
        date={2011},
     journal={Electron. J. Combin.},
      volume={18},
      number={1},
       pages={P84},
}

\bib{survey_combinatorial_testing}{article}{
      author={Nie, Changhai},
      author={Leung, Hareton},
       title={A survey of combinatorial testing.},
        date={2011},
     journal={{ACM} Comput. Surv.},
      volume={43},
       pages={1\ndash 29},
}

\bib{binomial_coeff_bounds}{article}{
      author={St\u{a}nic\u{a}, Pantelimon},
       title={Good lower and upper bounds on binomial coefficients},
        date={2001},
     journal={JIPAM. J. Inequal. Pure Appl. Math.},
      volume={2},
      number={3},
       pages={Paper No. 30, 5 p., electronic only},
}

\bib{ca_asymptotics_godbole}{article}{
      author={Yuan, Ruyue},
      author={Koch, Zoe},
      author={Godbole, Anant},
       title={Covering array bounds using analytical techniques},
        date={2014-05-12},
      eprint={arXiv: 1405.2844},
         url={http://arxiv.org/abs/1405.2844},
        note={Accessed on Oct 23, 2014.},
}

\end{biblist}
\end{bibdiv}
\end{document}